\documentclass[12pt]{amsart}

\usepackage{amsthm}
\usepackage{amssymb}
\usepackage{amsmath}
\usepackage{easywncy}

\numberwithin{equation}{section}
\newtheorem{thm}{Theorem}[section]
\newtheorem{prop}[thm]{Proposition}
\newtheorem{lem}[thm]{Lemma}
\newtheorem{cor}[thm]{Corollary}

\theoremstyle{definition}

\newtheorem{defn}[thm]{Definition}
\newtheorem{rem}[thm]{Remark}

\newcommand{\ba}{a}
\newcommand{\bb}{b}
\newcommand{\be}{\mathbf{e}}
\newcommand{\bk}{\mathbf{k}}
\newcommand{\bl}{\mathbf{l}}
\newcommand{\bx}{x}
\newcommand{\by}{y}
\newcommand{\bz}{z}
\newcommand{\blam}{\boldsymbol{\lambda}}
\newcommand{\bmu}{\boldsymbol{\mu}}
\newcommand{\shp}{\,\mathcyr{sh}\,}

\begin{document}

\title[Derivations for $q$-series and Applications]
{Derivations on the algebra of multiple harmonic $q$-series and their applications}
\author{Yoshihiro Takeyama}
\address{Department of Mathematics, 
Faculty of Pure and Applied Sciences, 
University of Tsukuba, Tsukuba, Ibaraki 305-8571, Japan}
\email{takeyama@math.tsukuba.ac.jp}
\thanks{
This is a pre-print of an article published in The Ramanujan Journal. 
The final authenticated version is available online at: 
\texttt{https://doi.org/10.1007/s11139-019-00139-y}. 
}

\begin{abstract}
We introduce derivations on the algebra of multiple harmonic $q$-series 
and show that they generate linear relations among the $q$-series 
which contain the derivation relations for a $q$-analogue of multiple zeta values due to Bradley. 
As a byproduct we obtain Ohno-type relations for 
finite multiple harmonic $q$-series at a root of unity.  
\end{abstract}
\maketitle

\setcounter{section}{0}
\setcounter{equation}{0}


\section{Introduction}

In this article we introduce derivations on the algebra of multiple harmonic $q$-series 
and show that they generate linear relations among the $q$-series, 
which are a slight generalization of Bradley's result in \cite{Bradley}. 
As a byproduct we obtain Ohno-type relations for 
finite multiple harmonic $q$-series at a root of unity 
introduced by Bachmann, Tasaka and the author in \cite{BTT}. 

For a tuple $\bk=(k_{1}, \ldots , k_{r})$ of positive integers with $k_{1} \ge 2$, 
the \textit{multiple zeta value} (MZV) $\zeta(\bk)$ is defined by 
\begin{align*}
\zeta(\bk)=\sum_{m_{1}>\cdots >m_{r} \ge 1}\frac{1}{m_{1}^{k_{1}}\cdots m_{r}^{k_{r}}}.  
\end{align*}
In \cite{IKZ} Ihara, Kaneko and Zagier introduce derivations on the algebra of MZVs and 
described linear relations among MZVs by using them (see also \cite{HO}). 
The precise statement is as follows. 
Let $\mathfrak{h}=\mathbb{Q}\langle x, y \rangle$ be the non-commutative polynomial ring, 
and set $z_{k}=x^{k-1}y$ for $k \ge 1$. 
Denote by $\mathfrak{h}^{0}$ the $\mathbb{Q}$-submodule of $\mathfrak{h}$ spanned by 
$1$ and the monomials $z_{k_{1}} \cdots z_{k_{r}}$ with $k_{1} \ge 2$. 
Then the $\mathbb{Q}$-linear map $Z: \mathfrak{h}^{0} \to \mathbb{R}$ is uniquely defined 
by $Z(1)=1$ and $Z(z_{k_{1}}\cdots z_{k_{r}})=\zeta(k_{1}, \ldots , k_{r})$. 
Now define the $\mathbb{Q}$-linear derivation $\partial_{n} \, (n \ge 1)$ on $\mathfrak{h}$ by 
$\partial_{n}(\bx)=-\partial_{n}(\by)=\bx(\bx+\by)^{n-1}\by$. 
Then it holds that 
$Z(\partial_{n}(w))=0$ for any $n \ge 1$ and $w \in \mathfrak{h}^{0}$. 
This gives linear relations among MZVs, which are called the derivation relations. 

In \cite{Bradley} Bradley proves that the derivation relations hold also for 
a $q$-analogue of MZVs defined by 
\begin{align}
\zeta_{q}(\bk)=\sum_{m_{1}>\cdots >m_{r}\ge 1}\frac{q^{(k_{1}-1)m_{1}+\cdots +(k_{r}-1)m_{r}}}
{[m_{1}]^{k_{1}} \cdots [m_{r}]^{k_{r}}},  
\label{eq:intro-qMZV-BZ}
\end{align}
where $q$ is a parameter satisfying $|q|<1$ and $[m]=(1-q^{m})/(1-q)$.  
The $q$-analogue model \eqref{eq:intro-qMZV-BZ} is often called the Bradley-Zhao model, 
which is a special value of the $q$-zeta function 
studied by Kaneko, Kurokawa and Wakayama \cite{KKW} in the case of $r=1$  
and Zhao \cite{Zhao} for $r \ge 2$. 
In \cite{Bradley} Bradley establishes that 
the model \eqref{eq:intro-qMZV-BZ} satisfies Ohno's relation \cite{Ohno} for MZVs in the same form, 
and then the derivation relations are obtained as a corollary. 

In this article we extend the derivation relations to 
a more general class of multiple harmonic $q$-series. 
Denote by $\mathcal{Z}_{q}$ the $\mathbb{Q}$-vector space spanned by 
the $q$-series of the following form 
\begin{align}
(1-q)^{s}\sum_{m_{1}>\cdots >m_{r}\ge 1}
\frac{q^{l_{1}m_{1}+\cdots +l_{r}m_{r}}}{[m_{1}]^{k_{1}} \cdots [m_{r}]^{k_{r}}}, 
\label{eq:gen-qMZV}
\end{align}
where $s \in \mathbb{Z}, l_{1} \ge 1$ and 
$k_{j}\ge 1, k_{j}\ge l_{j}\ge 0$ for any $j$. 
Note that the space $\mathcal{Z}_{q}$ contains 
the Bradley-Zhao model and other various $q$-analogue models due to, 
for example, Schlesinger \cite{Schlesinger}, 
Ohno-Okuda-Zudilin \cite{OOZ, Zudilin}, Okounkov \cite{Okounkov} and Bachmann-K\"uhn \cite{BK}.  

The proof of our derivation relations runs almost parallel to that in \cite{IKZ}. 
We define a subalgebra $\widehat{\mathfrak{H}^{0}}$ of 
a non-commutative polynomial ring with two indeterminates 
and the map $Z_{q}$ which sends an element of $\widehat{\mathfrak{H}^{0}}$ to $\mathcal{Z}_{q}$. 
As for the algebra $\mathfrak{h}^{0}$ of MZVs we can define two commutative products 
called the stuffle product and the shuffle product, 
and establish the double shuffle relations. 
Then we define the derivations $\partial_{n} \, (n \ge 1)$ on the algebra $\widehat{\mathfrak{H}^{0}}$ 
which relate the two products,  
and we obtain the derivation relations on $\mathcal{Z}_{q}$ from the double shuffle relations. 

In this paper, as a byproduct of the above construction,    
we also prove Ohno-type relations for finite multiple harmonic $q$-series at a root of unity 
Let $n$ be a positive integer and $\zeta_{n}$ a primitive $n$-th root of unity.  
Set 
\begin{align*}
z_{n}(\bk; \zeta_{n})=\sum_{n>m_{1}>\cdots >m_{r} \ge 1}
\frac{q^{(k_{1}-1)m_{1}+\cdots +(k_{r}-1)m_{r}}}{[m_{1}]^{k_{1}} \cdots [m_{r}]^{k_{r}}} 
\bigg|_{q=\zeta_{n}}.   
\end{align*} 
In \cite{BTT} Bachmann, Tasaka and the author found that 
the value $z_{n}(\bk; \zeta_{n})$ reproduces the finite multiple zeta values (FMZVs) 
and the symmetric multiple zeta values (SMZVs) introduced by 
Kaneko and Zagier \cite{KZ} 
through an algebraic operation and an analytic one, respectively. 
Kaneko and Zagier conjecture that there exists a $\mathbb{Q}$-algebra isomorphism 
which sends FMZVs to SMZVs (see Section \ref{subsec:FMHqS} for the precise statement).  
The above-mentioned property of $z_{n}(\bk; \zeta_{n})$ offers an explanation, though not a proof, 
for the conjecture. 

In \cite{Oyama} Oyama proves the Ohno-type relations for FMZVs and SMZVs  
by making use of the double shuffle relations among them and 
the derivations on the algebra $\mathfrak{h}$. 
In a similar manner we can prove Ohno-type relations for $z_{n}(\bk; \zeta_{n})$ 
by using the derivations defined in this paper. 
Our relations turn into the relations due to Oyama 
through the algebraic and analytic operation. 
Moreover, we also obtain Ohno-type relations for 
the cyclotomic analogue of FMZVs introduced in \cite{BTT}. 

The paper is organized as follows. 
In Section \ref{sec:DS} we formulate the double shuffle relations for 
multiple harmonic $q$-series. 
We prove the derivation relations on the space $\mathcal{Z}_{q}$ in Section \ref{sec:derivation}. 
In Section \ref{sec:Ohno} we show the Ohno-type relations for 
finite multiple harmonic $q$-series at a root of unity $z_{n}(\bk; \zeta_{n})$.


\section{Double shuffle relations for multiple harmonic $q$-series}\label{sec:DS}

\subsection{Preliminaries}

Let $\mathbb{N}$ be the set of positive integers. 
We introduce the letter $\overline{1}$ and set 
$\widehat{\mathbb{N}}=\{\overline{1}\}\sqcup \mathbb{N}=\{\overline{1}, 1, 2, \ldots \}$. 
Throughout this paper 
we call an ordered set $\mathbf{k}=(k_{1}, \ldots , k_{r})$ of elements of $\widehat{\mathbb{N}}$ 
an \textit{index}. 
An empty set $\emptyset$ is regarded as an index with $r=0$. 
We denote by $\widehat{I}$ the set of indices: 
\begin{align*}
\widehat{I}=
\{(k_{1}, \ldots , k_{r}) \,| \, r \ge 0 \,\, \hbox{and} \,\, k_{1}, \ldots , k_{r} \in \widehat{\mathbb{N}} \}.  
\end{align*}
We will use the subsets $I, \widehat{I}_{0}$ and $I_{0}$ of $\widehat{I}$ defined by 
\begin{align*}
& 
I=\{(k_{1}, \ldots , k_{r}) \in \widehat{I} \,| \, r \ge 0 \,\, \hbox{and} \,\, 
k_{j}\not=\overline{1} \, \hbox{for any $j$}\}, \\ 
& 
\widehat{I}_{0}=\{(k_{1}, \ldots , k_{r}) \in \widehat{I} \,| \, r \ge 0 \,\, \hbox{and} \,\, 
k_{1}\not=1 \}, \\ 
& 
I_{0}=I \cap \widehat{I}_{0}. 
\end{align*}
Note that the empty index $\emptyset$ belongs to all the three sets above. 

Fix a parameter $q \in \mathbb{C}$. 
For $m\ge 1$ we set 
\begin{align*}
F_{\overline{1}}(m)=\frac{q^{m}}{[m]}, \qquad 
F_{k}(m)=\frac{q^{(k-1)m}}{[m]^{k}} \quad (k\ge 1),  
\end{align*}
where $[m]=(1-q^{m})/(1-q)$ is the $q$-integer. 

Now assume that $|q|<1$. 
For a non-empty index $\mathbf{k}=(k_{1}, \ldots , k_{r})$ which belongs to $\widehat{I}_{0}$, 
we set 
\begin{align}
\zeta_{q}(\bk)=\sum_{m_{1}>\cdots >m_{r} \ge 1}\prod_{j=1}^{r}F_{k_{j}}(m_{j}).  
\label{eq:def-qMZV}
\end{align}
It converges absolutely because 
$|I_{k}(m)|\le |1-q||q|^{m}/(1-|q|)$ for any 
$k \in \widehat{\mathbb{N}}\setminus\{1\}$ and $m\ge 1$.
If $\bk \in I_{0}$, it is often called the \textit{Bradley-Zhao model of 
a $q$-analogue of multiple zeta values} \cite{Bradley, Zhao}. 
Hereafter we consider a more general class of $q$-series of the form \eqref{eq:def-qMZV} 
containing the factor $F_{\overline{1}}(m)=q^{m}/[m]$.  
By definition we set $\zeta_{q}(\emptyset)=1$.  

For a non-empty index $\bk=(k_{1}, \ldots , k_{r}) \in \widehat{I}$ 
we define the \textit{multiple polylogarithm of one variable} $L_{\bk}(t)$ by 
\begin{align*}
L_{\bk}(t)=\sum_{m_{1}>\cdots >m_{r} \ge 1}t^{m_{1}}\prod_{j=1}^{r}F_{k_{j}}(m_{j}).  
\end{align*}
We set $L_{\emptyset}(t)=1$. 
Note that $L_{\bk}(1)=\zeta_{q}(\bk)$ if $\bk \in \widehat{I}_{0}$.  

Let $\hbar$ be a formal variable and set $\mathcal{C}=\mathbb{Q}[\hbar, \hbar^{-1}]$. 
Denote by $\mathfrak{H}$ the non-commutative polynomial ring over $\mathcal{C}$ 
with two indeterminates $\ba$ and $\bb$.  
Set 
\begin{align*}
e_{\overline{1}}=\ba\bb, \qquad 
e_{k}=\ba^{k-1}(\ba+\hbar)\bb \quad (k \ge 1). 
\end{align*}
We denote by $\widehat{\mathfrak{H}^{1}}$ 
the subalgebra freely generated by the set $\{e_{k}\}_{k \in \widehat{\mathbb{N}}}$. 
Hereafter we will also use the element $e_{\overline{n}} \, (n \in \mathbb{N})$ of 
$\widehat{\mathfrak{H}^{1}}$ defined by 
\begin{align}
e_{\overline{n}}=\ba^{n}\bb=
\sum_{j=2}^{n}(-\hbar)^{n-j}e_{j}+(-\hbar)^{n-1}e_{\overline{1}}. 
\label{eq:def-enbar}
\end{align}
For a non-empty index $\bk=(k_{1}, \ldots , k_{r})$ we define 
$e_{\bk}=e_{k_{1}}\cdots e_{k_{r}}$. 
For the empty index we set $e_{\emptyset}=1$. 

We denote by $\widehat{\mathfrak{H}^{0}}$ the $\mathcal{C}$-submodule of $\mathfrak{H}$ 
spanned by the monomials $e_{\bk}$ with $\bk \in \widehat{I}_{0}$. 
We endow $\mathbb{C}$ with $\mathcal{C}$-module structure such that 
$\hbar$ acts as multiplication by $1-q$, 
and define the $\mathcal{C}$-linear map $Z_{q}: \widehat{\mathfrak{H}^{0}} \to \mathbb{C}$ by 
\begin{align*}
Z_{q}(e_{\bk})=\zeta_{q}(\bk) 
\end{align*}
for any $\bk \in \widehat{I}_{0}$. 
Since 
\begin{align*}
&
\frac{q^{lm}}{[m]^{k}}=\sum_{j=1}^{k-l}\binom{k-l}{j-1}(1-q)^{k-l-j}F_{l+j}(m) \qquad (k>l\ge 0),  \\ 
& 
\frac{q^{km}}{[m]^{k}}=\sum_{j=2}^{k}(q-1)^{k-j}F_{j}(m)+(q-1)^{k-1}F_{\overline{1}}(m)
\end{align*}
for $k \ge 1$ and $m \ge 1$, 
we see that the image $Z_{q}(\widehat{\mathfrak{H}^{0}})$ is equal to the $\mathbb{Q}$-vector space 
$\mathcal{Z}_{q}$ spanned by the $q$-series of the form \eqref{eq:gen-qMZV}.

\subsection{Double shuffle relations}

The double shuffle relations for multiple harmonic $q$-series $\zeta_{q}(\bk)$ are 
described in terms of two commutative and associative multiplication 
called the \textit{stuffle product} $*_{q}$ and the \textit{shuffle product} $\shp_{\! q}$. 
Here we recall their definitions and basic properties. 
See, e.g., \cite{CEM, Zhao2} for details. 

Denote by $\mathfrak{z}$ the $\mathcal{C}$-submodule of $\mathfrak{H}$ spanned by the set  
$\{e_{k}\}_{k \in \widehat{\mathbb{N}}}$. 
For $m\ge 1$ it holds that 
\begin{align*}
& 
F_{\overline{1}}(m)^{2}=F_{2}(m)-(1-q)F_{\overline{1}}(m), \qquad 
F_{\overline{1}}(m)F_{k}(m)=F_{k+1}(m) \quad (k \ge 1), \\ 
& 
F_{k}(m)F_{l}(m)=F_{k+l}(m)+(1-q)F_{k+l-1}(m) \quad (k, l\ge 1). 
\end{align*}
Motivated by the above properties we define 
the $\mathcal{C}$-bilinear symmetric map $\circ_{q}: \mathfrak{z}\times \mathfrak{z} \to \mathfrak{z}$ by  
\begin{align*}
e_{\overline{1}}\circ_{q} e_{\overline{1}}=e_{2}-\hbar \, e_{\overline{1}}, \quad 
e_{\overline{1}}\circ_{q} e_{k}=e_{k}\circ_{q} e_{\overline{1}}=e_{k+1}, \quad 
e_{k}\circ_{q} e_{l}=e_{k+l}+\hbar \, e_{k+l-1} 
\end{align*}
for $k, l\ge 1$. 
Note that 
$e_{\overline{1}} \circ_{q} e_{\overline{n}}=e_{\overline{n+1}}$ 
for $n \ge 1$, 
where $e_{\overline{n}}$ is defined by  \eqref{eq:def-enbar}. 

The stuffle product $*_{q}$ is the $\mathcal{C}$-bilinear binary operation on $\widehat{\mathfrak{H}^{1}}$ 
uniquely determined by 
\begin{align*}
& 
1*_{q}w=w*_{q}1=1, \\ 
& 
(e_{k}w)*_{q}(e_{l}w')=e_{k}(w*_{q}e_{l}w')+e_{l}(e_{k}w*_{q}w')+(e_{k}\circ_{q} e_{l})(w*_{q}w') 
\end{align*}
for $w, w' \in \widehat{\mathfrak{H}^{1}}$ and $k, l \in \widehat{\mathbb{N}}$. 
Note that the subalgebra $\widehat{\mathfrak{H}^{0}}$ is closed under the stuffle product $*_{q}$. 

\begin{prop}\label{prop:stuffle-rel}
For $w, w' \in \widehat{\mathfrak{H}^{0}}$ it holds that 
\begin{align*}
Z_{q}(w*_{q}w')=Z_{q}(w)Z_{q}(w'). 
\end{align*}
\end{prop}

Let $D_{r}=\{t \in \mathbb{C}\,|\,|t|<r\}$ be the open disc with radius $r$. 
Denote by $\mathcal{H}$ the $\mathbb{C}$-vector space of 
holomorphic functions $f$ on the unit disk $D_{1}$ satisfying the following condition: 
\begin{align*}
0<\forall{r}<1, \, \exists{C}>0, \, \forall{j}\in\mathbb{N}, \, \forall{t}\in \overline{D_{r}}\,\hbox{:}\, 
|f(q^{j}t)|\le C|q|^{j}. 
\end{align*}
We define the $\mathbb{Q}$-linear action of $\mathfrak{H}$ on $\mathcal{H}$ by 
$(\hbar f)(t)=(1-q)f(t)$ and 
\begin{align*}
(\ba f)(t)=(1-q)\sum_{j=1}^{\infty}f(q^{j}t), \qquad 
(\bb f)(t)=\frac{t}{1-t}f(t).  
\end{align*}
Note that $tf(t)/(1-t) \in \mathcal{H}$ for any bounded holomorphic function $f$ on $D_{1}$. 
Hence, for any index $\bk=(k_{1}, \ldots , k_{r})$, 
the polylogarithm $L_{\bk}$ is written as 
\begin{align*}
L_{\bk}=e_{k_{1}}\cdots e_{k_{r}}(\mathbf{1}),  
\end{align*}
where $\mathbf{1}$ is the constant function $\mathbf{1}(t)\equiv 1$. 
We define the $\mathcal{C}$-linear map 
$L: \widehat{\mathfrak{H}^{1}} \to \mathcal{F}+\mathbb{C}\mathbf{1}, \, w \mapsto L_{w}$ 
by $L_{e_{\bk}}=L_{\bk}$ for any index $\bk$. 

For $f, g \in \mathcal{F}$ it holds that 
\begin{align*}
(\ba f)(\ba g)=\ba\left((\ba f)g+f(\ba g)+\hbar fg\right), \qquad 
(\bb f)g=f(\bb g)=\bb(fg). 
\end{align*}
Motivated by the above properties we define 
the shuffle product $\shp_{\! q}$ as the $\mathcal{C}$-bilinear map 
$\shp_{\! q}: \mathfrak{H} \times \mathfrak{H}\to \mathfrak{H}$ by 
\begin{align*}
& 
1\shp_{\! q} w=w \shp_{\! q} 1=w, \\ 
& 
(\ba w)\shp_{\! q}(\ba w')=\ba ((\ba w)\shp_{\! q} w'+w\shp_{\! q}(\ba w')+\hbar\, w\shp_{\! q} w'), \\ 
& 
(\bb w)\shp_{\! q} w'=w\shp_{\! q} (\bb w')=\bb (w \shp_{\! q} w') 
\end{align*}
for $w, w' \in \mathfrak{H}$. 

\begin{prop}\label{prop:shuffle-rel}
\begin{enumerate}
 \item For $w, w' \in \widehat{\mathfrak{H}^{1}}$ it holds that $L_{w \shp_{\! q} w'}=L_{w}L_{w'}$.  
 \item The subalgebra $\widehat{\mathfrak{H}^{0}}$ is closed under the shuffle product $\shp_{\! q}$. 
 \item For $w, w' \in \widehat{\mathfrak{H}^{0}}$ it holds that $Z_{q}(w \shp_{\! q} w')=Z_{q}(w)Z_{q}(w')$. 
\end{enumerate} 
\end{prop}

Now we can state the double shuffle relations for multiple harmonic $q$-series, 
which follow from Proposition \ref{prop:stuffle-rel} and Proposition \ref{prop:shuffle-rel} (iii). 

\begin{thm}\label{thm:double-shuffle}
For any $w, w' \in \widehat{\mathfrak{H}^{0}}$ it holds that  
\begin{align*}
Z_{q}(w *_{q}w'-w \shp_{\! q} w')=0.  
\end{align*} 
\end{thm}

\section{Derivation relations}\label{sec:derivation} 

\subsection{Algebraic formulas of formal power series}

Denote by $\mathfrak{H}[[X]]$ the non-commutative ring of formal power series in $X$ 
whose coefficients belong to $\mathfrak{H}$. 
We extend the shuffle product on $\mathfrak{H}$ to $\mathfrak{H}[[X]]$ by 
$\mathcal{C}[[X]]$-linearity. 
Then we can define the logarithm and the exponential with respect to the shuffle product by 
\begin{align}
\log_{\shp_{\! q}}{(1+f(X))}=\sum_{n=1}^{\infty}\frac{(-1)^{n-1}}{n}(f(X))^{\shp_{\! q} n}, \quad 
\exp_{\shp_{\! q}}{(f(X))}=\sum_{n=0}^{\infty}\frac{1}{n!}(f(X))^{\shp_{\! q} n} 
\label{eq:def-log-exp}
\end{align}
for $f(X) \in X \mathfrak{H}[[X]]$, where   
$f(X)^{\shp_{\! q} 1}=f(X)$ and 
$f(X)^{\shp_{\! q} (n+1)}=f(X)\shp_{\! q}(f(X)^{\shp_{\! q} n})$ for $n \ge 1$. 
They are inverse to each other. 

Similarly we extend the stuffle product to the ring $\widehat{\mathfrak{H}^{1}}[[X]]$ 
of formal power series over $\widehat{\mathfrak{H}}^{1}$, 
and define the logarithm $\log_{*_{q}}{(1+f(X))}$ and the exponential $\exp_{*_{q}}{(f(X))}$ 
for $f(X)\in X \widehat{\mathfrak{H}^{1}}[[X]]$ by 
\eqref{eq:def-log-exp} with $\shp_{\! q}$ replaced by $*_{q}$.  

Now set 
\begin{align}
& 
\psi(X)=\frac{1}{\hbar}\ba\log{(1+\hbar \bb X)}=
\sum_{n=1}^{\infty}\frac{(-1)^{n-1}}{n}\hbar^{n-1}\ba\bb^{n}X^{n}, 
\label{eq:psi-def} \\ 
& 
\phi(X)=\left(\log{(1+\ba X)}\right)\bb=
\sum_{n=1}^{\infty}\frac{(-1)^{n-1}}{n}\ba^{n}\bb X^{n}. 
\label{eq:phi-def}
\end{align}
Note that they belong to the ideal $X \widehat{\mathfrak{H}^{1}}[[X]]$ because 
$\hbar^{n-1}\ba\bb^{n}=e_{\overline{1}}(e_{1}-e_{\overline{1}})^{n-1}$ and 
$\ba^{n}\bb=e_{\overline{n}}$ for $n \ge 1$.   
We will make use of the following formulas. 

\begin{prop}\label{prop:log-shp-*}
It holds that 
\begin{align}
& 
\psi(X)=\log_{\shp_{\! q}}{\left(\frac{1}{1-e_{\overline{1}}X}\right)}, 
\label{eq:log-shp} \\ 
& 
\phi(X)=\log_{*_{q}}{\left(\frac{1}{1-e_{\overline{1}}X}\right)}. 
\label{eq:log-*}
\end{align} 
\end{prop}  

\begin{proof}
First we prove \eqref{eq:log-shp}. 
Since the both sides belong to $X\mathfrak{H}[[X]]$, 
it suffices to show that the derivatives with respect to $X$ are equal, that is,   
\begin{align}
\frac{d}{dX}\left(\frac{1}{1-e_{\overline{1}}X}\right)=
\frac{1}{1-e_{\overline{1}}X}\shp_{\! q} \left(\ba\bb\frac{1}{1+\hbar \bb X}\right). 
\label{eq:proof-log-shp}
\end{align}
Rewrite the right hand side as follows. 
\begin{align*}
& 
\left(1+\ba\bb X\frac{1}{1-e_{\overline{1}}X}\right)\shp_{\! q} \left(\ba\bb\frac{1}{1+\hbar \bb X}\right) \\ 
&=\ba\bb\frac{1}{1+\hbar \bb X}+\ba \bb X\left\{ 
\frac{1}{1-e_{\overline{1}}X} \shp_{\! q} \left( \ba\bb \frac{1}{1+\hbar \bb X} \right)+
\frac{1}{1+\hbar \bb X}(\ba+\hbar)\bb \frac{1}{1-e_{\overline{1}}X}
\right\}.   
\end{align*}
Thus we find that 
\begin{align*}
(1-\ba \bb X)\left\{\frac{1}{1-e_{\overline{1}}X}\shp_{\! q} \left(\ba\bb\frac{1}{1+\hbar \bb X}\right) \right\}
&=\ba\bb\frac{1}{1+\hbar \bb X}
\left\{ 
1+(\ba+\hbar)\bb\frac{X}{1-e_{\overline{1}}X}
\right\} \\ 
&=\frac{e_{\overline{1}}}{1-e_{\overline{1}}X}. 
\end{align*}
This completes the proof of \eqref{eq:proof-log-shp}. 

The formula \eqref{eq:log-*} is proved in much the same way. 
It suffices to prove that 
\begin{align*}
\frac{d}{dX}\left(\frac{1}{1-e_{\overline{1}}X}\right)=
\frac{1}{1-e_{\overline{1}}X}*_{q}\frac{d}{dX}\phi(X).  
\end{align*}
The right hand side is equal to 
\begin{align*}
\sum_{n=1}^{\infty}(-X)^{n-1}\,  
\frac{1}{1-e_{\overline{1}}X}*_{q}e_{\overline{n}}. 
\end{align*}
For $n \ge 1$ it holds that 
\begin{align}
& 
\frac{1}{1-e_{\overline{1}}X}*_{q}e_{\overline{n}}=
\left(1+e_{\overline{1}}X\frac{1}{1-e_{\overline{1}}X}\right)*_{q}e_{\overline{n}}  
\label{eq:proof-log-*} \\ 
&=e_{\overline{n}}+X\left\{ 
e_{\overline{1}}\left( \frac{1}{1-e_{\overline{1}}X}*_{q}e_{\overline{n}} \right)+
e_{\overline{n}}e_{\overline{1}}\frac{1}{1-e_{\overline{1}}X}+(e_{\overline{1}}\circ_{q} e_{\overline{n}})
\frac{1}{1-e_{\overline{1}}X}
\right\}.
\nonumber 
\end{align}
Since $e_{\overline{1}}\circ_{q} e_{\overline{n}}=e_{\overline{n+1}}$ we see that 
\begin{align*}
(1-e_{\overline{1}}X) \left(\frac{1}{1-e_{\overline{1}}X}*_{q}e_{\overline{n}}\right)&=
e_{\overline{n}}+(Xe_{\overline{n}}e_{\overline{1}}+e_{\overline{n+1}})\frac{1}{1-e_{\overline{1}}X} \\ 
&=(e_{\overline{n}}+Xe_{\overline{n+1}})\frac{1}{1-e_{\overline{1}}X}. 
\end{align*}
Therefore 
\begin{align*}
\frac{1}{1-e_{\overline{1}}X}*_{q}\frac{d}{dX}\phi(X)&=
\frac{1}{1-e_{\overline{1}}X}\left( 
\sum_{n=1}^{\infty}(-X)^{n-1}(e_{\overline{n}}+Xe_{\overline{n+1}})
\right) 
\frac{1}{1-e_{\overline{1}}X} \\ 
&=\frac{e_{\overline{1}}}{(1-e_{\overline{1}}X)^{2}}=
\frac{d}{dX}\left(\frac{1}{1-e_{\overline{1}}X}\right). 
\end{align*}
\end{proof}

\subsection{Derivations} 

Here we define three derivations $\delta_{n}, d_{n}$ and $\partial_{n} \, (n \ge 1)$ 
on $\mathfrak{H}$, and prove a relation among them. 
They are regarded as counterparts of the derivations on the algebra of multiple zeta values 
due to Ihara, Kaneko and Zagier \cite{IKZ}. 

\begin{defn}
Let $\delta_{n}:\mathfrak{H} \to \mathfrak{H} \, (n\ge 1)$ be 
the $\mathcal{C}$-linear derivation uniquely determined by 
\begin{align*}
\delta_{n}(\ba)=0, \qquad 
\delta_{n}(\bb)=\frac{(-1)^{n-1}}{n}(\bb+1)\ba^{n}\bb. 
\end{align*}
We define the $\mathcal{C}$-algebra homomorphism 
$\Phi_{X}: \mathfrak{H} \to \mathfrak{H}[[X]]$ by 
\begin{align*}
\Phi_{X}=\exp{\left(\sum_{n=1}^{\infty}X^{n}\delta_{n}\right)}.  
\end{align*}
\end{defn}

We prove another representation of $\Phi_{X}$ on the subalgebra $\widehat{\mathfrak{H}^{1}}$. 

\begin{lem}
Set 
\begin{align}
D^{*}_{X}=\sum_{n=1}^{\infty}X^{n}\delta_{n}.  
\label{eq:def-DX}
\end{align}
For any $w \in \widehat{\mathfrak{H}^{1}}$ it holds that 
\begin{align*}
D^{*}_{X}(w)=\phi(X)*_{q}w-\phi(X)w,  
\end{align*}
where $\phi(X) \in X\widehat{\mathfrak{H}^{1}}[[X]]$ is given by \eqref{eq:phi-def}. 
\end{lem}

\begin{proof}
Define the map $\delta_{n}': \widehat{\mathfrak{H}^{1}} \to \widehat{\mathfrak{H}^{1}}$ by 
\begin{align*}
\delta_{n}'(w)=\frac{(-1)^{n-1}}{n}\left(e_{\overline{n}}*_{q}w-e_{\overline{n}}w\right).  
\end{align*} 
By direct calculation we see that it is a $\mathcal{C}$-linear derivation and  
$\delta_{n}'(e_{k})=\delta_{n}(e_{k})$ for $k \in \widehat{\mathbb{N}}$. 
Hence $\delta_{n}'=\delta_{n}$ on the subalgebra $\widehat{\mathfrak{H}^{1}}$. 
Using $\phi(X)=\sum_{n=1}^{\infty}(-1)^{n-1}e_{\overline{n}}X^{n}/n$, we get the desired formula. 
\end{proof}

\begin{prop}\label{prop:Phi-*-rewrite}
For $w \in \widehat{\mathfrak{H}^{1}}$, it holds that 
\begin{align*}
\Phi_{X}(w)=(1-e_{\overline{1}}X)\left(\frac{1}{1-e_{\overline{1}}X}*_{q}w \right).  
\end{align*}
\end{prop}

\begin{proof}
Denote by $\varphi_{X}$ the $\mathcal{C}$-linear map given in the right hand side above, that is,   
\begin{align*}
\varphi_{X}(w)=(1-e_{\overline{1}}X)\left(\frac{1}{1-e_{\overline{1}}X}*_{q}w \right).   
\end{align*}
By a similar calculation to that in the proof of Proposition \ref{prop:log-shp-*}
(see \eqref{eq:proof-log-*}), we find that 
\begin{align}
\varphi_{X}(e_{k})=(e_{k}+(e_{\overline{1}}\circ_{q} e_{k})X)\frac{1}{1-e_{\overline{1}}X} 
\label{eq:proof-Phi-*-rewrite}
\end{align}
and $\varphi_{X}(e_{k}w)=\varphi_{X}(e_{k})\varphi_{X}(w)$ for 
$k \in \widehat{\mathbb{N}}$ and $w \in \widehat{\mathfrak{H}^{1}}$. 
Hence $\varphi_{X}(ww')=\varphi_{X}(w)\varphi_{X}(w')$ for any $w, w' \in \widehat{\mathfrak{H}^{1}}$, 
and it suffices to show 
$\Phi_{X}(e_{k})=\varphi_{X}(e_{k})$ for any $k \in \widehat{\mathbb{N}}$. 

Let $D^{*}_{X}: \mathfrak{H}[[X]]\to \mathfrak{H}[[X]]$ be the operator defined by \eqref{eq:def-DX}, 
where $\delta_{n}$ is extended to $\mathfrak{H}[[X]]$ by $\mathcal{C}[[X]]$-linearity. 
Then $D^{*}_{X}(\ba)=0$ and, by induction on $s \ge 1$, we see that 
\begin{align*}
(D^{*}_{X})^{s}(\bb)=(\bb+1)\left(\phi(X)\right)^{*s},   
\end{align*}
where $\phi(X)$ is given by \eqref{eq:phi-def}. 
{}From the definition of $\Phi_{X}$ and \eqref{eq:log-*}, 
we find that 
\begin{align}
\Phi_{X}(\ba)=\ba, \quad 
\Phi_{X}(\bb)=(\bb+1)\exp_{*}(\phi(X))-1=(1+\ba X)\bb\frac{1}{1-\ba\bb X}.  
\label{eq:proof-Phi-*-rewrite2}
\end{align}
Using the above formulas, we see that $\Phi_{X}(e_{k})$ is equal to 
the right hand side of \eqref{eq:proof-Phi-*-rewrite} for any $k \in \widehat{\mathbb{N}}$.  
\end{proof}

Next we consider derivations associated with the shuffle product $\shp_{\! q}$. 

\begin{defn}
Let $d_{n}:\mathfrak{H}[[X]] \to \mathfrak{H}[[X]] \, (n \ge 1)$ be 
the $\mathcal{C}[[X]]$-linear derivation defined by 
\begin{align*}
d_{n}(w)=\frac{(-\hbar)^{n-1}}{n}
\left\{(\ba\bb^{n})\shp_{\! q} w-\ba\bb^{n} w \right\}. 
\end{align*}
We define the $\mathcal{C}[[X]]$-algebra homomorphism 
$\Psi_{X}: \mathfrak{H}[[X]] \to \mathfrak{H}[[X]]$ by 
\begin{align*}
\Psi_{X}=\exp{\left(\sum_{n=1}^{\infty}X^{n}d_{n}\right)}.  
\end{align*}
\end{defn}

As we will see below the operator $\Psi_{X}$ also has another representation, 
but this time on $\mathfrak{H}[[X]]$.  

\begin{lem}\label{lem:DX-shp}
Set 
\begin{align*}
D_{X}^{\shp}=\sum_{n=1}^{\infty}X^{n}d_{n}.  
\end{align*} 
Define $\rho_{s}(X) \in \widehat{\mathfrak{H}^{1}}[[X]] \, (s \ge 1)$ by the recurrence relation 
\begin{align*}
& 
\rho_{1}(X)=1, \\ 
& 
\rho_{s+1}(X)=\left(\psi(X)+\log{(1+\hbar\bb X)}\right)\rho_{s}(X)+\psi(X)\shp_{\! q} \rho_{s}(X) 
\quad (s \ge 1),   
\end{align*}
where $\psi(X)$ is given by \eqref{eq:psi-def}. 
\begin{enumerate}
 \item For $s \ge 1$ and $w \in \mathfrak{H}[[X]]$ it holds that 
\begin{align*}
(D_{X}^{\shp})^{s}(w)=\left(\psi(X)\rho_{s}(X)\right)\shp_{\! q} w-\psi(X)\left(\rho_{s}(X)\shp_{\! q} w\right).  
\end{align*}
 \item For $s \ge 1$ it holds that $\psi(X)\rho_{s}(X)=\left(\psi(X)\right)^{\shp_{\! q} s}$. 
\end{enumerate}
\end{lem}

\begin{proof}
By induction on $s$.  
\end{proof}

\begin{prop}\label{prop:Phi-shp-rewrite}
For $w \in \mathfrak{H}[[X]]$, it holds that 
\begin{align*} 
\Psi_{X}(w)=(1-e_{\overline{1}}X)\left(\frac{1}{1-e_{\overline{1}}X} \shp_{\! q} w \right).  
\end{align*}
\end{prop}

\begin{proof}
Let $\varphi_{X}$ be the $\mathcal{C}[[X]]$-linear map on $\mathfrak{H}[[X]]$ 
defined by 
\begin{align*}
\varphi_{X}(w)=(1-e_{\overline{1}}X)\left(\frac{1}{1-e_{\overline{1}}X} \shp_{\! q} w \right).   
\end{align*}
A similar calculation to that in the proof of Proposition \ref{prop:log-shp-*} shows that 
\begin{align}
\varphi_{X}(\ba)=\ba (1+\hbar \bb X)\frac{1}{1-\ba\bb X}, \quad 
\varphi_{X}(\bb)=(1-\ba \bb X)\bb \frac{1}{1-\ba\bb X},   
\label{eq:proof-Phi-shp-rewrite}
\end{align}
and that $\varphi_{X}$ is a $\mathcal{C}[[X]]$-algebra homomorphism. 
Therefore it suffices to show that 
$\varphi_{X}(u)=\Psi_{X}(u)$ for $u \in \{\ba, \bb\}$. 

First we calculate $\Psi_{X}(\ba)$. 
{}From Lemma \ref{lem:DX-shp} and the definition of $\Psi_{X}$  we see that 
\begin{align*}
\Psi_{X}(\ba)=\ba+\sum_{s=1}^{\infty}\frac{1}{s!}(\ba+\hbar)\psi(X)\rho_{s}(X)=
\ba+(\ba+\hbar)\left(\exp_{\shp_{\! q}}{(\psi(X))}-1\right). 
\end{align*}
Because of \eqref{eq:log-shp} it is equal to 
\begin{align*}
\ba+(\ba+\hbar)\left(\frac{1}{1-\ba\bb X}-1\right)=\ba(1+\hbar \bb X)\frac{1}{1-\ba\bb X}=
\varphi_{X}(\ba).   
\end{align*}

Next we calculate $\Psi_{X}(\bb)$. 
Using Proposition \ref{prop:log-shp-*} and \eqref{eq:log-shp} we see that 
\begin{align*}
\Psi_{X}(\bb)=\bb\exp_{\shp_{\! q}}{\left(\psi(X)\right)}-\ba \, \eta(X)=\bb\frac{1}{1-\ba\bb X}-\ba \bb \, \eta(X),   
\end{align*}
where $\eta(X)$ is given by 
\begin{align*}
\eta(X)=\frac{1}{\hbar}\log{(1+\hbar \bb X)}\sum_{s=1}^{\infty}\frac{1}{s!}\rho_{s}(X).  
\end{align*}
Since 
\begin{align*}
\ba \, \eta(X)=\psi(X)\sum_{s=1}^{\infty}\frac{1}{s!}\rho_{s}(X)=
\exp_{\shp_{\! q}}{(\psi(X))}-1=\ba \bb X\frac{1}{1-\ba\bb X}  
\end{align*}
and the map $w \mapsto \ba w$ of left multiplication by $\ba$ is injective on $\mathfrak{H}[[X]]$, 
we find that $\eta(X)=\bb X (1-\ba \bb X)^{-1}$. 
Thus we see that 
\begin{align*}
\Psi_{X}(\bb)=\bb\frac{1}{1-\ba\bb X}-\ba \bb^{2} X \frac{1}{1-\ba\bb X}=
(1-\ba \bb X)\bb \frac{1}{1-\ba\bb X}=\varphi_{X}(\bb). 
\end{align*}
This completes the proof. 
\end{proof}

Lastly we define the derivations $\partial_{n}$ on $\mathfrak{H}$ which 
relate $\delta_{n}$ and $d_{n}$. 

\begin{defn}
For $n \ge 1$ we define the $\mathcal{C}$-linear derivation 
$\partial_{n}: \mathfrak{H} \to \mathfrak{H}$ by 
\begin{align}
& 
\partial_{n}(\ba)=\frac{(-1)^{n}}{n}\ba \left\{ \ba(\bb+1)+\hbar\bb \right\}^{n-1}(\ba+\hbar)\bb, 
\label{eq:def-partial1} \\ 
& 
\partial_{n}(\bb)=\frac{(-1)^{n-1}}{n}\ba \left\{ (\bb+1)\ba+\hbar \bb\right\}^{n-1}(\bb+1)\bb. 
\label{eq:def-partial2} 
\end{align} 
We also define the $\mathcal{C}$-algebra homomorphism 
$\Delta_{X}: \mathfrak{H} \to \mathfrak{H}[[X]]$ by 
\begin{align*}
\Delta_{X}=\exp{\left(\sum_{n=1}^{\infty}X^{n}\partial_{n}\right)}.  
\end{align*}
\end{defn}

\begin{rem}
Using 
\begin{align*}
& 
\{\ba(\bb+1)+\hbar\bb\}(\ba+\hbar)=(\ba+\hbar)\{(\bb+1)\ba+\hbar\bb\}, \\ 
& 
\{(\bb+1)\ba+\hbar \bb\}(\bb+1)=(\bb+1)\{\ba(\bb+1)+\hbar \bb\},  
\end{align*} 
we see that 
\begin{align}
& 
\partial_{n}(\ba)=\frac{(-1)^{n}}{n}\ba (\ba+\hbar)
\left\{ (\bb+1)\ba+\hbar\bb \right\}^{n-1}\bb, 
\label{eq:def-partial3} \\ 
& 
\partial_{n}(\bb)=\frac{(-1)^{n-1}}{n}\ba(\bb+1) \left\{ \ba(\bb+1)+\hbar \bb\right\}^{n-1}\bb. 
\label{eq:def-partial4} 
\end{align}
\end{rem}

\begin{thm}\label{thm:Delta}
It holds that $\Phi_{X}=\Psi_{X}\Delta_{X}$ on $\mathfrak{H}$. 
\end{thm}

\begin{proof}
It suffices to show that 
$\Phi_{X}(u)=\Psi_{X}(\Delta_{X}(u))$ for $u\in \{\ba, \bb\}$. 
For that purpose we extend the map $\Delta_{X}$ to the quotient field of $\mathfrak{H}$ by 
$\Delta_{X}(w^{-1})=-w^{-1}\Delta_{X}(w)w^{-1}$ for $w \in \mathfrak{H}\setminus\{0\}$, 
and calculate $\Delta_{X}(\ba)$ and $\Delta_{X}(\bb)$.  
Set $z=\ba(\bb+1)+\hbar\bb$. 
{}From \eqref{eq:def-partial2} and \eqref{eq:def-partial3}, we see that 
\begin{align*}
\partial_{n}(z)=\partial_{n}(z+\hbar)=
\partial_{n}((\ba+\hbar)(\bb+1))=0.    
\end{align*}
Hence $\Delta_{X}(z)=z$. 
Moreover, using \eqref{eq:def-partial1} we find that 
\begin{align*}
\partial_{n}(\ba^{-1}-z^{-1})=\partial_{n}(\ba^{-1})=
\frac{(-1)^{n-1}}{n}z^{n}(\ba^{-1}-z^{-1})  
\end{align*}
for $n \ge 1$. 
Therefore
\begin{align*}
\Delta_{X}(\ba^{-1}-z^{-1})=\left(\exp{(\log{(1+zX)})}\right)(\ba^{-1}-z^{-1})=
(1+zX)(\ba^{-1}-z^{-1}).  
\end{align*}
Since $\Delta_{X}(z^{-1})=z^{-1}$ we get 
\begin{align}
\Delta_{X}(\ba)=(\Delta_{X}(\ba^{-1}))^{-1}=\ba\frac{1}{1+(\ba+\hbar)\bb X}.  
\label{eq:proof-Delta1} 
\end{align}
Combining \eqref{eq:proof-Delta1} and $\Delta_{X}(z)=z$, we find that 
\begin{align}
\Delta_{X}(\bb)=\left\{1+(\ba+(\ba+\hbar)\bb)X\right\}\bb\frac{1}{1+\hbar \bb X}.  
\label{eq:proof-Delta2} 
\end{align} 
{}From \eqref{eq:proof-Delta1}, \eqref{eq:proof-Delta2} and 
\eqref{eq:proof-Phi-shp-rewrite} with $\varphi_{X}$ replaced by $\Psi_{X}$, 
we see that $\Psi_{X}(\Delta_{X}(u))$ is equal to 
$\Phi_{X}(u)$ given by 
\eqref{eq:proof-Phi-*-rewrite2} for $u \in \{\ba, \bb\}$. 
\end{proof}

As a corollary of Proposition \ref{prop:Phi-*-rewrite}, Proposition \ref{prop:Phi-shp-rewrite} and 
Theorem \ref{thm:Delta}, we obtain the following relation, 
which plays a crucial role in the rest of this paper.   

\begin{cor}\label{cor:Delta}
For $w \in \widehat{\mathfrak{H}^{1}}$, it holds that 
\begin{align}
\frac{1}{1-e_{\overline{1}}X}*_{q}w=\frac{1}{1-e_{\overline{1}}X}\shp_{\! q} \Delta_{X}(w).   
\label{eq:cor-Delta}
\end{align} 
\end{cor}

\subsection{Derivation relations for multiple harmonic $q$-series}

Making use of \eqref{eq:cor-Delta} we prove the derivation relations 
for multiple harmonic $q$-series. 

\begin{lem}
The $\mathcal{C}$-subalgebra $\widehat{\mathfrak{H}^{0}}$ is invariant under 
the derivation $\partial_{n} \, (n \ge 1)$. 
\end{lem}

\begin{proof}
It suffices to prove that $\partial_{n}(e_{k}) \in \widehat{\mathfrak{H}^{0}}$ 
for $n \ge 1$ and $k \in \widehat{\mathbb{N}}$. 
Note that the operator of left multiplication by $\ba$ leaves $\widehat{\mathfrak{H}^{0}}$ invariant  
and the formula \eqref{eq:def-partial1} implies that the element  
\begin{align}
\partial_{n}(\ba)=\frac{(-1)^{n}}{n}\ba(\ba+e_{1})^{n-1}e_{1} 
\label{eq:partial-a}
\end{align}
belongs to $\widehat{\mathfrak{H}^{0}}$. 
Hence 
\begin{align}
\partial_{n}(e_{1})=\partial_{n}(z-\ba)=-\partial_{n}(\ba) \in \widehat{\mathfrak{H}^{0}},  
\label{eq:partial-ab}
\end{align}
where $z=\ba(\bb+1)+\hbar\bb$. 
{}From it we see that 
$\partial_{n}(e_{k})$ belongs to $\widehat{\mathfrak{H}^{0}}$ for $k \ge 2$ by induction on $k$ 
using 
\begin{align}
\partial_{n}(e_{k})=\partial_{n}(\ba \, e_{k-1})=\partial_{n}(\ba)e_{k-1}+\ba \,\partial_{n}(e_{k-1}).  
\label{eq:partial-rec}
\end{align}
Moreover, from \eqref{eq:def-partial4}, we have 
\begin{align*}
\partial_{n}(e_{1}-e_{\overline{1}})=\hbar \partial_{n}(\bb)=
\frac{(-1)^{n-1}}{n}(\ba+e_{\overline{1}})(\ba+e_{1})^{n-1}(e_{1}-e_{\overline{1}}),  
\end{align*}
which belongs to $\widehat{\mathfrak{H}^{0}}$, 
and hence also does $\partial_{n}(e_{\overline{1}})$.  
\end{proof}

Now we are in a position to prove the derivation relations. 

\begin{thm}\label{thm:derivation}
For any $n \ge 1$ and $w \in \widehat{\mathfrak{H}^{0}}$, 
it holds that $Z_{q}(\partial_{n}(w))=0$.  
\end{thm}

\begin{proof}
Corollary \ref{cor:Delta} implies that 
\begin{align*}
e_{\overline{1}}^{n}*w-e_{\overline{1}}^{n}\shp_{\! q} w=
\sum_{r=1}^{n} \frac{1}{r!} 
\sum_{\substack{m\ge 0, j_{1}, \ldots , j_{r}\ge 1 \\ m+j_{1}+\cdots +j_{r}=n}} 
e_{\overline{1}}^{m}\shp_{\! q} \left(\partial_{j_{1}}\cdots \partial_{j_{r}}(w)\right)
\end{align*} 
for $n \ge 1$ and $w \in \widehat{\mathfrak{H}^{0}}$. 
Using the double shuffle relations (Theorem \ref{thm:double-shuffle}) we see that 
\begin{align*}
\sum_{r=1}^{n} \frac{1}{r!} 
\sum_{\substack{m\ge 0, j_{1}, \ldots , j_{r}\ge 1 \\ m+j_{1}+\cdots +j_{r}=n}} 
Z_{q}(e_{\overline{1}}^{m})Z_{q}(\partial_{j_{1}}\cdots \partial_{j_{r}}(w))=0. 
\end{align*}
Now the induction on $n$ implies that $Z_{q}(\partial_{n}(w))=0$ for any $w \in \widehat{\mathfrak{H}^{0}}$. 
\end{proof}

Let us compare our derivation relations (Theorem \ref{thm:derivation}) with those 
for multiple zeta values. 
Let $\mathfrak{h}=\mathbb{Q}\langle \bx, \by \rangle$ be the non-commutative polynomial ring 
with indeterminates $\bx$ and $\by$.  
Set $\bz_{k}=\bx^{k-1}\by \, (k \ge 1)$ and $\mathfrak{h}^{0}=\mathbb{Q}+\bx \mathfrak{h}\by$. 
An index which belongs to $I_{0}$ is said to be \textit{admissible}. 
For an admissible index $\bk=(k_{1}, \ldots , k_{r})$, 
the \textit{multiple zeta value} (MZV) $\zeta(\bk)$ is defined by 
\begin{align*}
\zeta(\bk)=\sum_{m_{1}>\cdots >m_{r}\ge 1}\frac{1}{m_{1}^{k_{1}} \cdots m_{r}^{k_{r}}}.  
\end{align*}
Then the $\mathbb{Q}$-linear map $Z:\mathfrak{h}^{0} \to \mathbb{R}$ is uniquely determined by 
$Z(1)=1$ and $Z(\bz_{k_{1}} \cdots \bz_{k_{r}})=\zeta(k_{1}, \ldots , k_{r})$ for 
any admissible index $(k_{1}, \ldots , k_{r})$.  
Now define the $\mathbb{Q}$-linear derivation $\tilde{\partial}_{n} \,(n \ge 1)$ on $\mathfrak{h}$ by  
\begin{align*}
\tilde{\partial}_{n}(\bx)=\bx(\bx+\by)^{n-1}\by, \quad 
\tilde{\partial}_{n}(\by)=-\bx(\bx+\by)^{n-1}\by. 
\end{align*}
Then it holds that 
\begin{align*}
Z(\tilde{\partial}_{n}(w))=0 \qquad (n \ge 1, \, w \in \mathfrak{h}^{0}),  
\end{align*}
which is called the \textit{derivation relation} for MZVs \cite{HO, IKZ}. 

Set $\mathfrak{h}^{1}=\mathbb{Q}+\mathfrak{h}\by$. 
It is a $\mathbb{Q}$-algebra freely generated by the set 
$\{\bz_{k}\}_{k \ge 1}$. 
Hence it is embedded into $\widehat{\mathfrak{H}^{0}}$ through the $\mathbb{Q}$-algebra homomorphism 
$\iota: \mathfrak{h}^{1} \to \widehat{\mathfrak{H}^{0}}$ defined 
by $\iota(\bz_{k})=e_{k} \, (k \ge 1)$. 
{}From \eqref{eq:partial-a}, \eqref{eq:partial-ab} and \eqref{eq:partial-rec} we see that 
\begin{align*}
\frac{(-1)^{n}}{n}\iota\tilde{\partial}_{n}=\partial_{n}\iota 
\end{align*}
on $\mathfrak{h}^{1}$.  
Therefore Theorem \ref{thm:derivation} implies that 
\begin{align*}
Z_{q}(\iota(\tilde{\partial}_{n}(w)))=0 \qquad (n \ge 1, w \in \mathfrak{h}^{0}).  
\end{align*}

For an admissible index $(k_{1}, \ldots , k_{r})$, the $q$-series 
\begin{align*}
Z_{q}(\iota(\bz_{k_{1}} \cdots \bz_{k_{r}}))=\sum_{m_{1}>\cdots >m_{r}\ge 1}
\frac{q^{(k_{1}-1)m_{1}+\cdots +(k_{r}-1)m_{r}}}{[m_{1}]^{k_{1}} \cdots [m_{r}]^{k_{r}}} 
\end{align*}
is nothing but the Bradley-Zhao model of a $q$-analogue of MZVs. 
Thus we obtain another proof for the following theorem due to Bradley \cite{Bradley}. 

\begin{cor}
The Bradley-Zhao model of a $q$-analogue of multiple zeta values satisfies  
the derivation relations for multiple zeta values in the same form. 
\end{cor}

\section{Ohno-type relations}\label{sec:Ohno} 

\subsection{Finite multiple harmonic $q$-series at a root of unity}\label{subsec:FMHqS}

In \cite{BTT} Bachmann, Tasaka and the author introduce 
finite multiple harmonic $q$-series at a root of unity 
and find a connection to finite multiple zeta values (FMZVs) 
and symmetric multiple zeta values (SMZVs). 
Here we briefly recall the results in \cite{BTT}. 

Suppose that $n \ge 2$ and $\zeta_{n}$ is a primitive $n$-th root of unity. 
For an index $\bk=(k_{1}, \ldots , k_{r})$ which belongs to $I$, we set 
\begin{align*}
z_{n}(\bk; \zeta_{n})=\sum_{n>m_{1}>\cdots >m_{r}\ge 1}\prod_{j=1}^{r}F_{k_{j}}(m_{j})\big|_{q=\zeta_{n}}  
\end{align*}
and call it the \textit{finite multiple harmonic $q$-series at a root of unity}. 
By definition we set $z_{n}(\bk)=0$ if $r\ge n$. 
Note that, if $n$ is a prime $p$, 
$z_{p}(\bk; \zeta_{p})$ belongs to the integer ring $\mathbb{Z}[\zeta_{p}]$ 
because $[m]|_{q=\zeta_{p}}$ is a cyclotomic unit for $0<m<p$. 

Now we recall the definition of FMZVs. 
Set 
\begin{align*}
\mathcal{A}=\prod_{\hbox{\scriptsize $p$:prime}}\mathbb{F}_{p}/
\bigoplus_{\hbox{\scriptsize $p$:prime}}\mathbb{F}_{p}. 
\end{align*}
It is endowed with a $\mathbb{Q}$-algebra structure by diagonal multiplication.  
An element of $\mathcal{A}$ is represented by a sequence $(a_{p})_{p}$ of elements of $\mathbb{F}_{p}$, 
and two elements $(a_{p})_{p}$ and $(b_{p})_{p}$ of $\mathcal{A}$ are equal if $a_{p}=b_{p}$ for 
all but a finite number of primes $p$. 

Let $\bk=(k_{1}, \ldots , k_{r})$ be an index which belongs to $I$. 
The FMZV $\zeta_{\mathcal{A}}(\bk)$ is the element of $\mathcal{A}$ defined by 
\begin{align*}
\zeta_{\mathcal{A}}(\bk)=\left( 
\sum_{p>m_{1}>\cdots >m_{r}\ge 1}\frac{1}{m_{1}^{k_{1}} \cdots m_{r}^{k_{r}}}
\quad \mathrm{mod} \,\, p
\right)_{p}.  
\end{align*}

Next we recall the definition of SMZVs. 
We define the stuffle product $*$ on $\mathfrak{h}^{1}=\mathbb{Q}+\mathfrak{h} \by$ by 
\begin{align*}
& 
1*w=w*1=w, \\ 
& 
(\bz_{k}w)*(\bz_{l}w')=\bz_{k}(w*\bz_{l}w')+\bz_{l}(\bz_{k}w*w')+\bz_{k+l}(w*w') 
\end{align*}
for $w, w' \in \mathfrak{h}^{1}$ and $k, l\ge 1$. 
We denote by $\mathfrak{h}^{1}_{*}$ 
the commutative $\mathbb{Q}$-algebra $\mathfrak{h}^{1}$ equipped with the multiplication $*$.  
Let $\mathcal{Z}=\sum_{\bk \in I}\mathbb{Q}\zeta(\bk)$ the $\mathbb{Q}$-vector space 
spanned by all the MZVs. 
Then there exists a unique $\mathbb{Q}$-algebra homomorphism 
\begin{align*}
R: \mathfrak{h}^{1}_{*} \longrightarrow \mathcal{Z}[T] 
\end{align*}
such that $R(e_{1})=T$ and $R(e_{\bk})=\zeta(\bk)$ for any admissible index $\bk$. 

For an index $\bk$ which belongs to $I$ we set $R_{\bk}(T)=R(e_{\bk})$. 
Then the SMZV is defined by 
\begin{align*}
\zeta_{\mathcal{S}}(k_{1}, \ldots , k_{r})=\sum_{i=1}^{r}(-1)^{k_{1}+\cdots +k_{i}}
R_{k_{i}, k_{i-1}, \ldots , k_{1}}(T)R_{k_{i+1}, \ldots , k_{r-1}, k_{r}}(T). 
\end{align*}
The right hand side does not depend on $T$ and belongs to $\mathcal{Z}$. 

In \cite{KZ} Kaneko and Zagier conjecture that there exists the $\mathbb{Q}$-algebra homomorphism 
\begin{align*}
\varphi: \mathcal{A} \longrightarrow \mathcal{Z}/\zeta(2)\mathcal{Z} 
\end{align*}
such that $\varphi(\zeta_{\mathcal{A}}(\bk))\equiv \zeta_{\mathcal{S}}(\bk)$ mod $\zeta(2)\mathcal{Z}$ 
for any $\bk \in I$. 

\begin{thm}\cite{BTT}
\begin{enumerate}
 \item For a prime $p$, we denote by $\mathfrak{p}_{p}$ 
the ideal of $\mathbb{Z}[\zeta_{p}]$ generated by $1-\zeta_{p}$, 
and identify $\mathbb{Z}[\zeta_{p}]/\mathfrak{p}_{p}$ with the finite field $\mathbb{F}_{p}$. 
Then, for any $\bk \in I$, it holds that 
\begin{align*}
(z_{p}(\bk; \zeta_{p}) \, \mathrm{mod} \, \mathfrak{p}_{p})_{p}=\zeta_{\mathcal{A}}(\bk)  
\end{align*}
in $\mathcal{A}$. 
 \item For any $\bk \in I$, the limit 
\begin{align}
\xi(\bk)=\lim_{n \to \infty}z_{n}(\bk; e^{2\pi i/n}) 
\label{eq:def-xi}
\end{align}
converges and it holds that $\mathop{\mathrm{Im}}\xi(\bk) \in \pi \mathcal{Z}$ and 
\begin{align*}
\mathop{\mathrm{Re}}{\xi(\bk)} \equiv \zeta_{\mathcal{S}}(\bk) \qquad \mathrm{mod}\,\, \zeta(2)\mathcal{Z}.  
\end{align*}
\end{enumerate} 
\end{thm}

\subsection{Double shuffle relations}

Here we state the double shuffle relations for 
the finite multiple harmonic $q$-series at a root of unity. 
The proof is similar to that for FMZVs (see \cite{Kaneko}). 

In this subsection we fix a primitive $n$-th root of unity $\zeta_{n}$ for $n\ge 2$. 
We denote by the same letter $z_{n}$ the $\mathcal{C}$-linear map 
$z_{n}: \widehat{\mathfrak{H}^{1}} \to \mathbb{C}$ uniquely determined by 
$z_{n}(e_{\bk})=z_{n}(\bk; \zeta_{n})$ for any $\bk \in I$, 
where the $\mathcal{C}$-module structure of $\mathbb{C}$ is defined by 
$\hbar c=(1-\zeta_{n})c$ for $c \in \mathbb{C}$. 
 
\begin{thm}\label{thm:DS}
\begin{enumerate}
 \item For $n \ge 2$ and $w, w' \in \widehat{\mathfrak{H}^{1}}$, 
it holds that 
\begin{align*}
z_{n}(w*_{q} w')=z_{n}(w)z_{n}(w').  
\end{align*}
 \item Let $\psi: \widehat{\mathfrak{H}}^{1} \to \widehat{\mathfrak{H}}^{1}$ be 
the $\mathcal{C}$-algebra anti-involution defined by 
\begin{align*}
\psi(e_{\overline{1}})=-e_{1}, \quad 
\psi(e_{1})=-e_{\overline{1}}, \quad 
\psi(e_{k})=(-1)^{k}\sum_{j=2}^{k}\binom{k-2}{j-2}\hbar^{k-j}e_{j} \,\, (k \ge 2).  
\end{align*}
Then it holds that 
\begin{align}
z_{n}(w \shp_{\! q} w')=z_{n}(\psi(w)w')  
\label{eq:cyc-DS}
\end{align}
for $n \ge 2$ and $w, w' \in \widehat{\mathfrak{H}^{1}}$. 
\end{enumerate}
\end{thm}

\begin{proof} 
The proof for (i) is similar to that of Proposition \ref{prop:stuffle-rel}. 
Here we prove (ii). 

For the time being we assume that $|q|<1$. 
For $m \ge 1$ we define the $\mathcal{C}$-linear map 
$A_{m}: \widehat{\mathfrak{H}^{1}} \to \mathbb{C}$ by
\begin{align*}
A_{m}(1)=1, \qquad 
A_{m}(e_{\bk})=\sum_{m=m_{1}>m_{2}>\cdots >m_{r}>0}\prod_{j=1}^{r}F_{k_{j}}(m_{j}) 
\end{align*}
for $\bk=(k_{1}, \ldots , k_{r}) \in \widehat{I}$. 
Then we have 
\begin{align*}
L_{w}(t)=\sum_{m=1}^{\infty}t^{m}A_{m}(w) \qquad (w \in \widehat{\mathfrak{H}^{1}}).  
\end{align*}
Proposition \ref{prop:shuffle-rel} implies that 
\begin{align*}
A_{m}(w_{1}\shp_{\! q} w_{2})=\sum_{\substack{\alpha+\beta=m \\ \alpha, \beta \ge 0}}
A_{\alpha}(w_{1})A_{\beta}(w_{2})  
\end{align*}
for $m \ge 1$ and $w_{1}, w_{2}\in\widehat{\mathfrak{H}^{1}}$. 
Note that it is an equality of rational functions in $q$ 
without poles at $n$-th roots of unity for $n>m$. 
Hence we can set $q=\zeta_{n}$. 

Now let us prove \eqref{eq:cyc-DS}. 
We may assume that 
$w=e_{\bk}$ and $w'=e_{\bl}$ for some indices 
$\bk=(k_{1}, \ldots , k_{r})$ and $\bl=(l_{1}, \ldots , l_{s})$. 
{}From the previous observation we see that 
\begin{align*}
z_{n}(e_{\bk}\shp_{\! q}e_{\bl})
&=\sum_{n>m>0}A_{m}(e_{\bk}\shp_{\! q} e_{\bl})\big|_{q=\zeta_{n}}=
\sum_{n>m>0}\sum_{\substack{\alpha+\beta=m \\ \alpha, \beta \ge 0}}
A_{\alpha}(e_{\bk})A_{\beta}(e_{\bl})\big|_{q=\zeta_{n}} \\ 
&=\sum_{n>\alpha>\beta \ge 0}
A_{n-\alpha}(e_{\bk})A_{\beta}(e_{\bl})\big|_{q=\zeta_{n}}. 
\end{align*}
When $q=\zeta_{n}$ it holds that 
\begin{align*}
& 
F_{\overline{1}}(n-m)=-F_{1}(m), \qquad 
F_{1}(n-m)=-F_{\overline{1}}(m), \\ 
& 
F_{k}(n-m)=(-1)^{k}\sum_{j=2}^{k}(1-\zeta_{n})^{k-j}\binom{k-2}{j-2}F_{j}(m) 
\end{align*}
for $n>m>0$ and $k \ge 2$. 
By changing the summation variable 
$m_{j}$ to $n-m_{r+1-j}$ in $A_{n-\alpha}(e_{\bk})$ and   
using the above formulas, 
we see that 
\begin{align*}
\sum_{n>\alpha>\beta \ge 0}
A_{n-\alpha}(e_{\bk})A_{\beta}(e_{\bl})\big|_{q=\zeta_{n}}=z_{n}(\psi(e_{\bk})e_{\bl}).   
\end{align*}
This completes the proof. 
\end{proof}

\subsection{Ohno-type relations}

In \cite{Oyama} Oyama proves linear relations for FMZVs and SMZVs 
of a similar form to Ohno's relations for MZVs \cite{Ohno}. 
Here we prove their $q$-analogue for the finite multiple harmonic $q$-series 
at a root of unity. 

Denote by $\mathfrak{H}^{1}_{\mathbb{Q}}$ the subalgebra of $\widehat{\mathfrak{H}^{1}}$ generated by 
the set $\{e_{k}\}_{k \ge 1}$ over $\mathbb{Q}$. 
Set $e_{0}=\ba$. 
Then we can identify $\mathfrak{H}^{1}_{\mathbb{Q}}$ with 
the non-commutative polynomial ring $\mathbb{Q}\langle e_{0}, e_{1} \rangle$. 
Note that $e_{k}=e_{0}^{k-1}e_{1}$ for $k \ge 1$. 

Let $\tau$ be the $\mathbb{Q}$-linear involution on $\mathfrak{H}^{1}_{\mathbb{Q}}$ defined by 
$\tau(e_{0})=e_{1}$ and $\tau(e_{1})=e_{0}$. 
Suppose that $\bk \in I\setminus\{\emptyset\}$. 
The monomial $e_{\bk}$ is uniquely written in the form 
$e_{\bk}=w e_{1}$ for some word $w$ in $\{e_{k}\}_{k \ge 1}$. 
Then we define the \textit{Hoffman dual} $\bk^{\vee}$ by 
the relation $\tau(w)e_{1}=e_{\bk^{\vee}}$. 
For example, if $\bk=(2,3,1)$, we have 
$e_{\bk}=(e_{0}e_{1}e_{0}^{2}e_{1})e_{1}$ and hence 
$e_{\bk^{\vee}}=\tau(e_{0}e_{1}e_{0}^{2}e_{1})e_{1}=e_{1}e_{0}e_{1}^{2}e_{0}e_{1}=e_{1}e_{2}e_{1}e_{2}$. 
Therefore $(2, 3, 1)^{\vee}=(1,2,1,2)$. 

For a tuple of non-negative integers $\be=(e_{1}, \ldots , e_{r})$ 
we define the \textit{depth} $\mathrm{dep}(\be)$ and the \textit{weight} $\mathrm{wt}(\be)$ by 
$\mathrm{dep}(\be)=r$ and $\mathrm{wt}(\be)=\sum_{j=1}^{r}e_{j}$, respectively.  

The Ohno-type relations are given as follows. 

\begin{thm}\label{thm:Ohno-rel} 
Suppose that $\bk \in I\setminus\{\emptyset\}$ and $\mathrm{dep}(\bk)=r$. 
Set $s=\mathrm{dep}(\bk^{\vee})=\mathrm{wt}(\bk)-r+1$. 
For $m \ge 0$ and $n \ge r+m+1$, it holds that 
\begin{align}
\sum_{\substack{\be \in (\mathbb{Z}_{\ge 0})^{s} \\ \mathrm{wt}(\be)=m}}
z_{n}((\bk^{\vee}+\be)^{\vee}; \zeta_{n})=\sum_{l=0}^{m}\frac{1}{n}\binom{n}{m-l+1}(1-\zeta_{n})^{m-l}
\sum_{\substack{\be' \in (\mathbb{Z}_{\ge 0})^{r} \\ \mathrm{wt}(\be')=l}}
z_{n}(\bk+\be'; \zeta_{n}).   
\label{eq:Ohno-rel} 
\end{align}
\end{thm}

\begin{proof}
The proof is similar to that given in \cite{Oyama}.  
{}From \eqref{eq:proof-Delta1} and \eqref{eq:proof-Delta2}, we see that 
\begin{align}
\Delta_{X}(e_{k})=\left(e_{0}\frac{1}{1+e_{1}X}\right)^{k-1}
e_{1}\left(1+e_{0}X\frac{1}{1+e_{1}X}\right) 
\qquad (k \ge 1). 
\label{eq:proof-Ohno1}
\end{align} 
For $\bk=(k_{1}, \ldots , k_{r}) \in I\setminus\{\emptyset\}$ and $s\ge 0$, 
we define $a_{s}(\bk)$ by 
\begin{align*}
a_{s}(\bk)=\sum \prod_{1\le i \le r}^{\curvearrowright}
\left[ \left( \prod_{1\le l \le k_{i}}^{\curvearrowright}(e_{0}e_{1}^{j_{l}^{(i)}}) \right)e_{1}\right],  
\end{align*}
where the sum in the right hand side is over the set 
\begin{align*}
\{ (j_{l}^{(i)})_{\substack{1\le i \le r \\ 1\le l \le k_{i}}} \in (\mathbb{Z}_{\ge 0})^{\mathrm{wt}(\bk)} \, |\, 
\sum_{i=1}^{r}\sum_{l=1}^{k_{i}}j_{l}^{(i)}=s \}
\end{align*}
and $\displaystyle \prod_{1\le i \le k}^{\curvearrowright}Y_{i}$ stands for 
the ordered product $Y_{1}\cdots Y_{k}$. 
We also set 
\begin{align*}
A_{\bk, s, p}=\sum_{\substack{\lambda_{1}, \ldots , \lambda_{r} \in \{0, 1\} \\ \lambda_{1}+\cdots +\lambda_{r}=p}}
a_{s}(k_{1}+\lambda_{1}-1, \ldots, k_{r}+\lambda_{r}-1) 
\end{align*}
for $\bk=(k_{1}, \ldots , k_{r}) \in I\setminus\{\emptyset\}$ and $s, p \ge 0$. 
Then the equality \eqref{eq:proof-Ohno1} implies that 
\begin{align*}
\Delta_{X}(e_{\bk})=\sum_{p, s\ge 0}(-1)^{s}X^{p+s}A_{\bk, s, p}.  
\end{align*}
{}From the above formula and Corollary \ref{cor:Delta} we find that 
\begin{align*}
e_{\overline{1}}^{m}*_{q}e_{\bk}=\sum_{\substack{l, s, p \ge 0 \\ l+s+p=m}}
(-1)^{s}e_{\overline{1}}^{l}\shp_{\! q}A_{\bk, s, p} 
\end{align*}
for $m \ge 0$. 
Hence Theorem \ref{thm:DS} implies that 
\begin{align*}
(-1)^{m}z_{n}(e_{\overline{1}}^{m})z_{n}(e_{\bk})=
\sum_{p=0}^{\mathrm{min}\{m, r\}}(-1)^{p}
\sum_{\substack{l, s \ge 0 \\ l+s=m-p}}z_{n}(e_{1}^{l}A_{\bk, s, p}) 
\end{align*}
for $m \ge 0$ and $n \ge 1$. 
Now use the following formula \cite[Lemma 2.4]{Oyama} 
\begin{align*}
\sum_{\substack{l, s \ge 0 \\ l+s=m-p}} e_{1}^{l}A_{\bk, s, p}=
\sum_{\substack{\blam \in \{0, 1\}^{r} \\ \mathrm{wt}(\blam)=p}}  
\sum_{\substack{\be \in (\mathbb{Z}_{\ge 0})^{d} \\ \mathrm{wt}(\be)=m-p}}
e_{((\bk+\blam)^{\vee}+\be)^{\vee}},   
\end{align*}
where $r=\mathrm{dep}(\bk)$ and $d=\mathrm{dep}((\bk+\blam)^{\vee})$. 
Then we obtain 
\begin{align}
(-1)^{m}z_{n}(\{\overline{1}\}^{m})z_{n}(\bk)=
\sum_{p=0}^{\mathrm{min}\{m, r\}}(-1)^{p}
\sum_{\substack{\blam \in \{0, 1\}^{r} \\ \mathrm{wt}(\blam)=p}}  
\sum_{\substack{\be \in (\mathbb{Z}_{\ge 0})^{d} \\ \mathrm{wt}(\be)=m-p}}
z_{n}(((\bk+\blam)^{\vee}+\be)^{\vee})   
\label{eq:proof-Ohno2}
\end{align}
for $m \ge 0$. 

Now we take $\bk \in I\setminus\{\emptyset\}$, and 
set $r=\mathrm{dep}(\bk)$ and $s=\mathrm{dep}(\bk^{\vee})$.  
We show that 
\begin{align}
\sum_{\substack{\be \in (\mathbb{Z}_{\ge 0})^{s} \\ \mathrm{wt}(\be)=m}}z_{n}((\bk^{\vee}+\be)^{\vee})=
\sum_{l=0}^{m}(-1)^{m-l}z_{n}(\{\overline{1}\}^{m-l})
\sum_{\substack{\be \in (\mathbb{Z}_{\ge 0})^{r} \\ \mathrm{wt}(\be)=l}}z_{n}(\bk+\be) 
\label{eq:proof-Ohno25}
\end{align}
for $0\le m \le n-r-1$ by induction on $m$.   
If $m=0$ it follows from $(\bk^{\vee})^{\vee}=\bk$. 
Suppose that $m \ge 1$. 
Separate the right hand side of \eqref{eq:proof-Ohno2} into two parts with $p=0$ and $p\ge 1$. 
Then, from the induction hypothesis, we see that 
\begin{align}
& 
\sum_{\substack{\be \in (\mathbb{Z}_{\ge 0})^{s} \\ \mathrm{wt}(\be)=m}}z_{n}((\bk^{\vee}+\be)^{\vee})=
(-1)^{m}z_{n}(\{\overline{1}\}^{m})z_{n}(\bk) 
\label{eq:proof-Ohno3} \\ 
&\quad {}+\sum_{l=1}^{m}(-1)^{m-l}z_{n}(\{\overline{1}\}^{m-l})
\sum_{p=1}^{l}(-1)^{p-1}
\sum_{\substack{\blam \in \{0, 1\}^{r} \\ \mathrm{wt}(\blam)=p}}
\sum_{\substack{\bmu \in (\mathbb{Z}_{\ge 0})^{r} \\ \mathrm{wt}(\bmu)=l-p}}
z_{n}(\bk+\blam+\bmu).  
\nonumber 
\end{align}
For $1 \le p \le l$ we define the map 
\begin{align*}
\nu_{p}: & \{\blam \in \{0, 1\}^{r}\,|\,\mathrm{wt}(\blam)=p\}\times 
\{\bmu \in (\mathbb{Z}_{\ge 0})^{r}\,|\,\mathrm{wt}(\bmu)=l-p\} \\ 
& \quad \to 
\{\be \in (\mathbb{Z}_{\ge 0})^{r}\,|\,\mathrm{wt}(\be)=l\} 
\end{align*}
by $\nu_{p}(\blam, \bmu)=\blam+\bmu$. 
Then it holds that 
\begin{align*}
\sum_{p=1}^{l}(-1)^{p-1}
\sum_{\substack{\blam \in \{0, 1\}^{r} \\ \mathrm{wt}(\blam)=p}}
\sum_{\substack{\bmu \in (\mathbb{Z}_{\ge 0})^{r} \\ \mathrm{wt}(\bmu)=l-p}}
z_{n}(\bk+\blam+\bmu)=
\sum_{\substack{\be \in (\mathbb{Z}_{\ge 0})^{r} \\ \mathrm{wt}(\be)=l}}z_{n}(\bk+\be)
\sum_{p=1}^{l}(-1)^{p-1}\left|\nu_{p}^{-1}(\be)\right|. 
\end{align*}
For a tuple $\be=(e_{1}, \ldots , e_{r})$ of non-negative integers, set 
\begin{align*}
\mathrm{supp}(\be)=\{j \in \{1, 2, \ldots , r\}\,|\, e_{j}\ge 1\}.  
\end{align*}
Then we see that 
\begin{align*}
\nu_{p}^{-1}(\be)=\{(\blam, \be-\blam) \, |\, \blam \in \{0, 1\}^{r}, \, \mathrm{wt(\blam)}=p, \,  
\mathrm{supp}(\blam) \subset \mathrm{supp}(\be)\}.   
\end{align*}
Hence 
\begin{align*}
\sum_{p=1}^{l}(-1)^{p-1}\left|\nu_{p}^{-1}(\be)\right|=
\sum_{p=1}^{l}(-1)^{p-1}\binom{|\mathrm{supp}(\be)|}{p}=1
\end{align*}
for $\be \in (\mathbb{Z}_{\ge 0})^{r}$ satisfying $\mathrm{wt}(\be)=l$. 
Therefore the right hand side of \eqref{eq:proof-Ohno3} is equal to that of \eqref{eq:proof-Ohno25}. 

To derive the desired equality from \eqref{eq:proof-Ohno25}, we should prove that 
\begin{align*}
z_{n}(\{\overline{1}\}^{r}; \zeta_{n})=\frac{(-1)^{r}}{n}\binom{n}{r+1}(1-\zeta_{n})^{r} 
\end{align*}
for $n>r \ge 0$. 
We obtain it by calculating the generating function 
\begin{align*}
& 
\sum_{r=0}^{n-1}z_{n}(\{\overline{1}\}^{r}; \zeta_{n})T^{r}=
\prod_{m=1}^{n-1}\left(1+\frac{q^{m}}{[m]}T\right)\bigg|_{q=\zeta_{n}}=
\prod_{m=1}^{n-1}\frac{1-\zeta_{n}^{m}(1-(1-\zeta_{n})T)}{1-\zeta_{n}^{m}} \\ 
&=\frac{1-(1-(1-\zeta_{n})T)^{n}}{n(1-\zeta_{n})T}=
\sum_{r=0}^{n-1}\frac{(-1)^{r}}{n}\binom{n}{r+1}(1-\zeta_{n})^{r}T^{r}. 
\end{align*}
This completes the proof. 
\end{proof}

As a corollary of Theorem \ref{thm:Ohno-rel} we reproduce the Ohno-type relations 
for FMZVs and SMZVs. 

\begin{cor}\cite{Oyama}
Suppose that $\bk \in I\setminus\{\emptyset\}$. 
Set $r=\mathrm{dep}(\bk)$ and $s=\mathrm{dep}(\bk^{\vee})$. 
Then it holds that 
\begin{align*}
\sum_{\substack{\be \in (\mathbb{Z}_{\ge 0})^{s} \\ \mathrm{wt}(\be)=m}}
\zeta_{\mathcal{F}}((\bk^{\vee}+\be)^{\vee})=
\sum_{\substack{\be' \in (\mathbb{Z}_{\ge 0})^{r} \\ \mathrm{wt}(\be)=m}}
\zeta_{\mathcal{F}}(\bk+\be')
\end{align*}
for $m \ge 0$ and $\mathcal{F}=\mathcal{A}$ or $\mathcal{S}$.  
\end{cor}

\begin{proof}
First we consider the case of FMZVs. 
If $n$ is a prime $p$, then 
$\frac{1}{p}\binom{p}{m-l+1}$ is an integer for $m\ge l \ge 0$. 
Therefore, taking modulo $(1-\zeta_{p})$ of both sides of \eqref{eq:Ohno-rel}, 
we obtain the desired relation for FMZVs. 

To consider the case of SMZVs, 
we set $\zeta_{n}=e^{2\pi i/n}$ in \eqref{eq:Ohno-rel} and calculate the limit as $n \to \infty$. 
Using Stirling's formula, we see that 
\begin{align*}
\frac{1}{n}\binom{n}{m-l+1}(1-e^{2\pi i/n})^{m-l} \to (-2\pi i)^{m-l} \qquad (n \to \infty)  
\end{align*}
for $m \ge l \ge 0$. 
Hence it holds that 
\begin{align*}
\sum_{\substack{\be \in (\mathbb{Z}_{\ge 0})^{s} \\ \mathrm{wt}(\be)=m}}
\xi((\bk^{\vee}+\be)^{\vee})=
\sum_{l=0}^{m}(-2\pi i)^{m-l}
\sum_{\substack{\be' \in (\mathbb{Z}_{\ge 0})^{r} \\ \mathrm{wt}(\be)=m}}
\xi(\bk+\be'),  
\end{align*}
where $\xi(\bk)$ is defined by \eqref{eq:def-xi}. 
Taking the real parts modulo $\zeta(2)\mathcal{Z}=\pi^{2}\mathcal{Z}$, 
we obtain the desired equality for SMZVs. 
\end{proof}

Finally we prove Ohno-type relations for 
the cyclotomic analogue of FMZVs introduced in \cite{BTT}. 
As an analogue of $\mathcal{A}$ we define 
\begin{align*}
\mathcal{A}^{\mathrm{cyc}}=\left(\prod_{\hbox{\scriptsize $p$:prime}}\mathbb{Z}[\zeta_{p}]/(p)\right)/
\left(\bigoplus_{\hbox{\scriptsize $p$:prime}}\mathbb{Z}[\zeta_{p}]/(p)\right).  
\end{align*}
It also carries the $\mathbb{Q}$-algebra structure. 
The \textit{cyclotomic analogue of FMZV} is defined by 
\begin{align*}
Z^{\mathrm{cyc}}(\bk)=\left( z_{p}(\bk; \zeta_{p}) \quad \mathrm{mod}\,\, (p) \right)_{p} 
\in \mathcal{A}^{\mathrm{cyc}} 
\end{align*}
for $\bk \in I$. 

To write down the Ohno-type relation we need the $\mathbb{Q}$-linear map $L$ defined as follows. 
We define the stuffle product $*$ on $\mathfrak{H}^{1}_{\mathbb{Q}}$ by 
\begin{align*}
& 
1*w=w*1=1, \\ 
& 
(e_{k}w)*(e_{l}w')=e_{k}(w*e_{l}w')+e_{l}(e_{k}w*w')+e_{k+l}(w*w')
\end{align*}
for $w, w' \in \mathfrak{H}^{1}_{\mathbb{Q}}$ and $k, l \ge 1$. 
We define the $\mathbb{Q}$-linear map $L: \mathfrak{H}^{1} \to \mathfrak{H}^{1}$ by 
\begin{align*}
L(e_{\bk})=-\frac{1}{2\mathrm{dep}(\bk)+1}e_{1}*e_{\bk} 
\end{align*}
for $\bk \in I$. 

Let $\mathbb{Q}I$ be the $\mathbb{Q}$-vector space with the basis $I$. 
By abuse of notation we denote by the same letter $L$ 
the $\mathbb{Q}$-linear transformation on $\mathbb{Q}I$
defined by $L(\bk)=\sum_{\bk'}a_{\bk, \bk'}\bk'$, 
where $a_{\bk,\bk'}$ is a rational number determined by  
$L(e_{\bk})=\sum_{\bk'}a_{\bk, \bk'}e_{\bk'}$. 
We also extend the map $Z^{\mathrm{cyc}}: I \to \mathcal{A}^{\mathrm{cyc}}$ to 
$\mathbb{Q}I$ by $\mathbb{Q}$-linearity. 

Set $\varpi=(1-\zeta_{p})_{p} \in \mathcal{A}^{\mathrm{cyc}}$. 
Then it holds that 
\begin{align}
\varpi Z^{\mathrm{cyc}}(\bk)=Z^{\mathrm{cyc}}(L(\bk))   
\label{eq:varpi-L}
\end{align}
(see \cite{BTT}). 

\begin{cor}
For $\bk \in I\setminus\{\emptyset\}$ and $m \ge 0$, it holds that 
\begin{align*}
\sum_{\substack{\be \in (\mathbb{Z}_{\ge 0})^{s} \\ \mathrm{wt}(\be)=m}}
Z^{\mathrm{cyc}}((\bk^{\vee}+\be)^{\vee})=
\sum_{l=0}^{m}\frac{(-1)^{m-l}}{m-l+1}
\sum_{\substack{\be' \in (\mathbb{Z}_{\ge 0})^{r} \\ \mathrm{wt}(\be)=l}}
Z^{\mathrm{cyc}}(L^{m-l}(\bk+\be')),    
\end{align*} 
where $r=\mathrm{dep}(\bk)$ and $s=\mathrm{dep}(\bk^{\vee})$. 
\end{cor}

\begin{proof}
It follows from Theorem \ref{thm:Ohno-rel}, \eqref{eq:varpi-L} and 
\begin{align*}
\frac{1}{p}\binom{p}{m-l+1}\equiv \frac{(-1)^{m-l}}{m-l+1} \qquad \mathrm{mod}\,\, p 
\qquad (m\ge l \ge 0)
\end{align*} 
for any prime $p$ such that $p\ge r+m+1$. 
\end{proof}


\section*{Acknowledgments}

The research of the author is supported by 
JSPS KAKENHI Grant Number 18K03233. 
The author is deeply grateful to Kojiro Oyama for explanation about \cite{Oyama}.  
He also wishes to express his gratitude to Professor Masanobu Kaneko 
for valuable information on \cite{Kaneko}. 


\end{document}